\documentclass{elsarticle}

\usepackage{amsmath,amscd,amsthm}
\usepackage{amssymb}
\usepackage{mathrsfs}
\usepackage{pstricks}
\usepackage{pst-plot}
\usepackage[left=1.1875in,top=1in,text={6.5in,9in}]{geometry}
\usepackage{enumerate}
\usepackage{graphicx}

\newtheorem{theorem}{Theorem}
\theoremstyle{plain}
\newtheorem{lemma}[theorem]{Lemma}
\newtheorem{prop}[theorem]{Proposition}
\newtheorem{cor}[theorem]{Corollary}

\theoremstyle{definition}
\newtheorem{definition}[theorem]{Definition}
\theoremstyle{remark}

\newtheorem{rmk}[theorem]{Remark}
\newtheorem{example}[theorem]{Example}
\newtheorem*{acknowledgement*}{Acknowledgements}

\newcommand{\Z}{\mathbb{Z}}
\newcommand{\ds}{\displaystyle}

\begin{document}

\begin{frontmatter}

\title{Erd\H{o}s-Szekeres Tableaux}

\author[]{Shaun V. Ault}\ead{svault@valdosta.edu}

\author{Benjamin Shemmer\corref{cor1}}\ead{bshemmer@fordham.edu}

\address{Department of Mathematics and Computer Science\\
  Valdosta State University \\
  Valdosta, GA, 31602, USA. \\
  \bigskip
  Department of Mathematics \\
  Fordham University \\
  Bronx, New York, 10461, USA.}

\cortext[cor1]{Corresponding author}

\begin{abstract}
  We explore a question related to the celebrated Erd\H{o}s-Szekeres
  Theorem and develop a geometric approach to answer it.  Our main
  object of study is the {\it Erd\H{o}s-Szekeres tableau}, or EST, of
  a number sequence.  An EST is the sequence of integral points whose
  coordinates record the length of the longest increasing and longest
  decreasing subsequence ending at each element of the sequence.  We
  define the {\it Order Poset} of an EST in order to answer the
  question: {\it What information about the sequence can be recovered
    by its EST?}
\end{abstract}


\end{frontmatter}


\section{Introduction}

There is a well-known result of Erd\H{o}s and Szekeres~\cite{ES} that
states:
\begin{theorem}[Erd\H{o}s-Szekeres]\label{thm.ES}
  Given a sequence of $n$ distinct real numbers, if $n
  > rs$, there exists a monotonically decreasing
  subsequence of length at least $r+1$ or a monotonically increasing
  subsequence of length at least $s+1$.
\end{theorem}
The original appears in the context of a geometric problem: {\it Can
  one find a function $N(n)$ such that any set containing $N(n)$
  distinct points in the plane has a subset of $n$ points that form a
  convex polygon?}  Thus, Theorem~\ref{thm.ES} may be considered a
work of {\it combinatorial geometry}, roughly defined as ``using
combinatorics to solve a problem of geometry.''  We would like to
contribute to the reverse direction of study, exploring what geometry
may tell us about certain problems in combinatorics.  For recent work
related to the Erd\H{o}s-Szekeres theorem, see~\cite{Myers,FPSS,EM}.

\subsection{Sequences and Monotone Subsequences}

We provide Seidenberg's proof of Theorem~\ref{thm.ES} in order to
introduce the main geometric construct with which we will be concerned.
  See~\cite{St} for alternate proofs.

\bigskip

\noindent {\it Theorem~\ref{thm.ES}} (Seidenberg~\cite{Seid})

\begin{proof}
Let $A = (a_1, a_2, \ldots, a_n)$ be a sequence of distinct real
numbers.  To each $a_i$ assign a pair of numbers $(a_i^+, a_i^-)$,
where $a_i^+$, {\it resp.}  $a_i^-$, is the length of the longest
increasing, {\it resp.}  decreasing, subsequence in $A$ that ends at
$a_i$.  We first show that for $i < j$, the pairs $(a_i^+, a_i^-)$ and
$(a_j^+, a_j^-)$ must be distinct.  Indeed, if $a_i < a_j$, then
$a_j^+ \geq a_i^+ + 1$, while if $a_i > a_j$, then $a_j^- \geq a_i^- +
1$.  This implies that there are $n$ distinct pairs.  Now suppose that
$n > rs$, for some natural numbers $r$ and $s$.  If it were the case
that all decreasing subsequences had length at most $r$ and all
increasing subsequences had length at most $s$, then at most $rs$ of
the pairs $(a_i^+, a_i^-)$ could be distinct.  Therefore, there is at
least one index $i$ such that either $a_i^- \geq r+1$ or $a_i^+ \geq
s+1$.
\end{proof}

\subsection{The Erd\H{o}s-Szekeres Tableau}

The geometric object of study in this paper is the sequence of the
points $\{(a_i^+, a_i^-)\}_{i \leq n}$ defined in the proof of
Theorem~\ref{thm.ES}.  Henceforth, the term {\it point} will always
refer to a first-quadrant lattice point $(x,y) \in (\Z^+)^2$.

\begin{definition}
  Given a sequence $A = (a_1, \ldots, a_n)$ of distinct numbers, the
  {\bf Erd\H{o}s-Szekeres tableau} (or, {\bf EST}) of $A$ is the
  sequence $T(A) = ( t_i )_{i \leq n}$ where $t_i = (a_i^+, a_i^-)$
  and $a_i^+$, {\it resp.}  $a_i^-$, is the length of the longest
  increasing, {\it resp.}  decreasing, subsequence in $A$ that ends at
  $a_i$.  An EST may be visualized as the set of points $\{(a_i^+,
  a_i^-) \;|\; 1 \leq i \leq n\}$ together with arrows or number
  labels that indicate their order in the sequence.  Given an EST $T=(
  t_i )_{i \leq n}$ and a number $m \leq n$, the subsequence $( t_i
  )_{i \leq m}$ is called a sub-EST of $T$ and denoted $T_m$.
\end{definition}

The points of an EST must be distinct, as discussed in the proof of
Theorem~\ref{thm.ES}.

\begin{example}\label{ex.EST}
  Let $A = (1,3,8,5,7,4,6,2)$.  Then $T(A) = ( (1,1), (2,1), (3,1),
  (3,2), (4,2), (3,3), (4,3), (2,4))$ (see
  Fig.~\ref{fig.EST_example}).
\end{example}

\begin{figure}[!ht]
  \centering
  \psset{unit=0.8cm}
  \begin{pspicture}(0,-1)(5,5)
    \psaxes[Dx=1,Dy=1,linecolor=black,linewidth=1pt]{->}(0,0)(0,0)(5,5)
    \psdot(1,1)
    \psdot(3,2)
    \psdot(2,1)
    \psdot(4,2)
    \psdot(3,3)
    \psdot(3,1)
    \psdot(4,3)
    \psdot(2,4)
    \psline[arrowscale=2]{->}(1,1)(2,1)
    \psline[arrowscale=2]{->}(2,1)(3,1)
    \psline[arrowscale=2]{->}(3,1)(3,2)
    \psline[arrowscale=2]{->}(3,2)(4,2)
    \psline[arrowscale=2]{->}(4,2)(3,3)
    \psline[arrowscale=2]{->}(3,3)(4,3)
    \psline[arrowscale=2]{->}(4,3)(2,4)
  \end{pspicture}
  \caption{EST for $A = (1,3,8,5,7,4,6,2)$.}\label{fig.EST_example}
\end{figure}
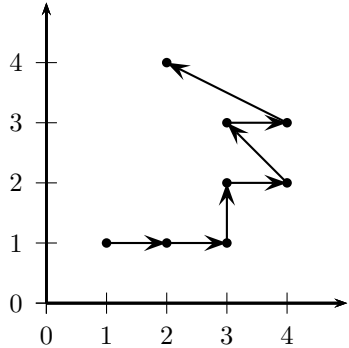

\begin{definition}
  The set of all ESTs of size $n$ will be denoted $\mathcal{T}_n$.
\end{definition}

As evidenced by Example~\ref{ex.EST}, the structure of a typical EST
is interesting and non-trivial.  In what follows, we will give a
concrete description of which sequences of points constitute a
valid EST, and conversely, which number sequences $A = (a_i)_{i \leq
  n}$ can produce a given EST.  The correspondence $A \to T(A)$ is not
one-to-one.  For example, both $A = (2, 1, 4, 3)$ and $B = (3, 1, 4,
2)$ have the same EST.  Our primary goal is to determine a necessary
and sufficient set of relations among the elements of $A$ such that
$T(A)$ is determined and to classify all sequences $B$ such that $T(B)
= T(A)$.  A secondary goal is to understand the size of
$\mathcal{T}_n$.

Certain structures in an EST prove to be important indicators of which
sequences of points are, in fact, ESTs.  The following terms will
apply to any sequence of points, whether it corresponds to an EST or
not.
\begin{definition}
  Let $T = (t_i)_{i \leq n} = \left((x_i, y_i)\right)_{i \leq n}$ be a
  sequence of distinct points.
  \begin{enumerate}
    \item The {\bf column above $t_i$} is defined by
      $\mathrm{Col}(t_i) = \mathrm{Col}(x_i,y_i) = \{ (x_i, y) \;|\; y
      > y_i\}$.
    \item The {\bf row beyond $t_i$} is defined by $\mathrm{Row}(t_i)
      = \mathrm{Row}(x_i,y_i) = \{ (x, y_i) \;|\; x > x_i\}$.
    \item The {\bf shadow} of $T$ is the set of points, $\ds{S(T) =
      \bigcup_{i \leq n} \left(\{ (x, y_i) \;|\; x \leq x_i\} \cup \{
      (x_i, y) \;|\; y \leq y_i\}\right)}$.
    \item The {\bf wall} of $T$ is the set of points, $\ds{W(T) =
      \left(S(T) + \{(1,0), (0,1)\}\right) \setminus S(T)}$, if $|T| >
      0$, and $\{ (1,1) \}$ if $|T| = 0$.
    \item An {\bf edge point} of $T$ is a member of the set, $E(T) =
      \{ t \in T \;|\; \mathrm{Col}(t) \cap T = \emptyset \vee
      \mathrm{Row}(t) \cap T = \emptyset\}$.
   \item A {\bf corner point} of $T$ is a member of the set,
      \[
        C(T) = \{ (x, y) \in E(T) \;|\; (x + 1, y) \notin S(T) \wedge
        (x, y + 1) \notin S(T)\}.
      \]
      When applied to $T_i$, we use the convenient notation: $S_i =
      S(T_i)$, $W_i = W(T_i)$, etc.
  \end{enumerate}
\end{definition}

\begin{example}
  Consider the Erd\H{o}s-Szekeres tableau $T = T(A)$ of
  Example~\ref{ex.EST}.  The sets $S(T)$, $W(T)$ and $C(T$) are
  depicted in Fig.~\ref{fig.shadow_wall_example}.  The shadow is
  represented by the light gray region (although only the points with
  positive integer coordinates within that region are part of the
  shadow), the wall by the dark gray regions (again, only the integral
  points), the edge points by large dots, and the
  corner points (which are also edge points) by large
  open dots.
\end{example}

\begin{figure}[!ht]
  \centering
  \psset{unit=0.8cm}
  \begin{pspicture}(0,-1)(5.5,6)
    \psframe*[linecolor=lightgray](0.5,0.5)(4.5,3.5)
    \psframe*[linecolor=lightgray](0.5,0.5)(2.5,4.5)
    \rput(1.5,2.5){$S(T)$}
    \psframe*[linecolor=gray](0.5,4.5)(2.5,5.5)
    \psframe*[linecolor=gray](2.5,3.5)(4.5,4.5)
    \psframe*[linecolor=gray](4.5,0.5)(5.5,3.5)
    \rput(2.8,5){$W(T)$}
    \psaxes[Dx=1,Dy=1,linecolor=black,linewidth=1pt]{->}(0,0)(0,0)(5.5,5.5)
    \psdot(2,1)
    \psdot(3,2)
    \psdot[linewidth=2pt](4,3)
    \psdot[linewidth=2pt](2,4)
    \psdot[linewidth=2pt](1,1)
    \psdot[linewidth=2pt](3,3)
    \psdot[linewidth=2pt](4,2)
    \psdot[linewidth=2pt](3,1)
    \psline[arrowscale=2]{->}(1,1)(2,1)
    \psline[arrowscale=2]{->}(2,1)(3,1)
    \psline[arrowscale=2]{->}(3,1)(3,2)
    \psline[arrowscale=2]{->}(3,2)(4,2)
    \psline[arrowscale=2]{->}(4,2)(3,3)
    \psline[arrowscale=2]{->}(3,3)(4,3)
    \psline[arrowscale=2]{->}(4,3)(2,4)
    \psdot[linewidth=0.25pt](1,1)
    \psdot[linewidth=0.25pt](1,2)
    \psdot[linewidth=0.25pt](1,3)
    \psdot[linewidth=0.25pt](1,4)
    \psdot[linewidth=0.25pt](2,1)
    \psdot[linewidth=0.25pt](2,2)
    \psdot[linewidth=0.25pt](2,3)
    \psdot[linewidth=0.25pt](2,4)
    \psdot[linewidth=0.25pt](3,1)
    \psdot[linewidth=0.25pt](3,2)
    \psdot[linewidth=0.25pt](3,3)
    \psdot[linewidth=0.25pt](4,1)
    \psdot[linewidth=0.25pt](4,2)
    \psdot[linewidth=0.25pt](4,3)
    \psdot[linewidth=0.25pt](1,5)
    \psdot[linewidth=0.25pt](3,4)
    \psdot[linewidth=0.25pt](2,5)
    \psdot[linewidth=0.25pt](4,4)
    \psdot[linewidth=0.25pt](5,3)
    \psdot[linewidth=0.25pt](5,2)
    \psdot[linewidth=0.25pt](5,1)
    \psdots[linewidth=1.5pt,linecolor=white](2,4)
    \psdots[linewidth=1.5pt,linecolor=white](4,3)
  \end{pspicture}
  \caption{$S(T)$, $W(T)$, $E(T)$ and $C(T)$ for $T =
    T((1,3,8,5,7,4,6,2))$.}
  \label{fig.shadow_wall_example}
\end{figure}
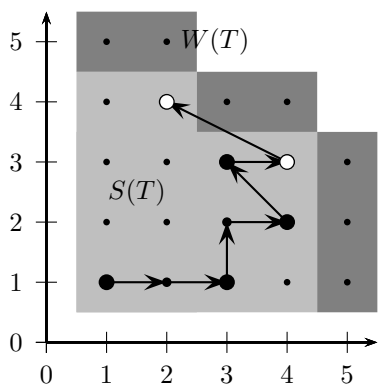

The following results follow from the definitions and will become
useful in later sections. The proofs are left as an exercises for the
reader.
\begin{lemma}\label{lem.structure_lemmas}
  \begin{enumerate}[(a)] For any $T \in \mathcal{T}_n$,
    \item $T_1 \subset T_2 \subset \cdots \subset T_n$ and $S_1
      \subset S_2 \subset \cdots \subset S_n$
    \item $C_n \subseteq E_n \subseteq E_{n-1} \cup \{t_n\} \subseteq
      T_n \subseteq S_n$
    \item $S_n \cap W_n = \emptyset$.
    \item $|W_n| = |E_n| + 1$.
  \end{enumerate}
\end{lemma}

Furthermore, there is a natural duality relationship concerning
sequences and their ESTs.  Let $A = (a_i)_{i \leq n}$ be a sequence of
real numbers, and define the dual of $A$ by $A^* = (-a_i)_{i \leq n}$.
If we specialize to sequences in $\{1, 2, \ldots, n\}$, then we may
define $A^* = (n+1-a_i)_{i \leq n}$. This is the same type of duality
as Myers uses in~\cite{Myers}.
\begin{definition}
  Given $T = T(A)$, the {\bf dual} EST of $T$ is $T^* = T(A^*)$.
\end{definition}
It should be clear that $T^*$ is obtained by reflecting the points of
$T$ across the diagonal line $y = x$, thus rows are exchanged with
columns and vice versa.

\subsection{The Order Poset}\label{sub.order_poset}

In this note, we do not concern ourselves with the actual numeric
values of $a_i$ in a sequence $A = (a_i)_{i \leq n}$.  Indeed, we only
are concerned with the relative order of the $a_i$.  That is, we
regard a sequence $A$ as a linear order on the set of indices, $( [n],
<_A)$, in which $i <_A j$ if and only if $a_i < a_j$ (here, $[n] =
\{1, 2, \ldots, n\}$).  It is natural to ask what order information is
preserved in $T(A)$.  We define an equivalence relation on the set of
sequences (linear orders of $[n]$) by $\{ (A, B) \;|\; T(A) = T(B)
\}$. For $T\in\mathcal{T}_n$, define $[T]=\{A:T(A)=T\}$.  Thus, to
answer what information about $A$ is preserved in $T(A)$, we should
examine the order relations common to all $B \in [T(A)]$.

\begin{definition}
  Let $T \in \mathcal{T}_n$.  Define a partial order $P(T) =
  \left([n], <_{P(T)}\right)$ by $i <_{P(T)} j$ if and only if $i <_A
  j$ for all sequences $A \in [T]$.  The poset $P(T)$ is called the
  {\bf order poset} of $T$.
\end{definition}

Certain relations in $P(T)$ are immediate from the configuration of
points of $T$.  It may come as no surprise that rows and columns of an
EST provide information about the order of the sequence that generated
the EST.

\begin{lemma}[Row and Column Relations]\label{lem.row_col_relations}
   Let $A = (a_i)_{i \leq n}$ be any sequence and $T(A) = (t_i)_{i
     \leq n}$.
   \begin{enumerate}[(a)]
     \item (Row Relation) If $t_j \in \mathrm{Row}(t_i)$, then $a_i <
       a_j$.
     \item (Column Relation) If $t_j \in \mathrm{Col}(t_i)$, then $a_i
       > a_j$.
   \end{enumerate}
\end{lemma}
\begin{proof}
  We need only prove the row relation, as the column relation would
  follow by duality in $T^*$.  Suppose that $a_i > a_j$.  Since $t_j
  \in \mathrm{Row}(t_i)$, we have $a_i^+ < a_j^+$ and $a_i^- = a_j^-$.
  Now if $a_j$ precedes $a_i$ (that is, $j < i$), then by appending
  $a_i$ to the longest increasing subsequence ending at $a_j$, we find
  $a_i^+ > a_j^+$, a contradiction.  On this other hand, if $a_i$
  precedes $a_j$ (that is, $i < j$), we may append $a_j$ to the
  longest decreasing subsequence ending at $a_i$, which would imply
  $a_j^- > a_i^-$, another contradiction.  Thus, $a_i < a_j$ is
  forced.
\end{proof}

The following result generalizes the row and column relations.
Because these relations may involve a turn around an empty column or
row, they are given the suggestive names {\it hook} and {\it slice}.
\begin{lemma}[Hook and Slice Relations]\label{lem.hook_slice}
   Let $A = (a_i)_{i \leq n}$ be any sequence and $T = T(A)$.  Write
   $T = (t_i) = ((x_i, y_i))$.
   \begin{enumerate}[(a)]
     \item (Hook Relation) Suppose $i < j$.  If there is a number $x$
       such that $x_i \leq x \leq x_j$ and $\mathrm{Col}(x,y_i) \cap
       T_j = \emptyset$, then $a_i < a_j$.
     \item (Slice Relation) Suppose $i < j$. If there is a number $y$
       such that $y_i \leq y \leq y_j$ and $\mathrm{Row}(x_i,y) \cap
       T_j = \emptyset$, then $a_i > a_j$.
   \end{enumerate}
\end{lemma}
\begin{proof}
  Suppose $i < j$ and we have a point $(x, a_i^-)$ such that $a_i^+
  \leq x \leq a_j^+$ and $\mathrm{Col}(x, a_i^-) \cap T_j =
  \emptyset$.  Suppose that $a_i > a_j$.  Let $(a_{k_1}, a_{k_2},
  \ldots, a_j)$ be an increasing subsequence
  ending at $a_j$ of length $a_j^+$.  Since $x \leq a_j^+$, we
  have $k_x \leq j$.  Suppose that $a_{k_x}$ preceded $a_i$.  Then,
  since $a_{k_x} \leq a_j < a_i$, we have $a_i^+ > a_{k_x}^+ = x$, a
  contradiction.  Thus $a_i$ must precede $a_{k_x}$.  But then
  $a_{k_x}^- > a_i^-$, which puts $a_{k_x} \in \mathrm{Col}(x, a_i^-)
  \cap T_{j}$, another contradiction.  Hence the statement $a_i > a_j$
  is false.  The slice relation follows from the hook relation in
  $T^*$ by duality.
\end{proof}

For example, in Fig.~\ref{fig.EST_example}, $t_2$ hooks $t_4$
since $\mathrm{Col}(2,1) \cap T_4 = \emptyset$ (the point $t_8 =
(2,4)$ belongs to $T_8$, not $T_4$, so the column is considered empty
in $T_4$). Therefore, $b_2 < b_4$ for {\it any} sequence $B$ that
generates this EST.  Also, $t_4$ slices $t_6$ since
$\mathrm{Row}(3,3) \cap T_6 = \emptyset$.  This implies $b_4 > b_6$
for any $B$.  Moreover, $t_4$ hooks $t_5$ since
$\mathrm{Col}(3,3) \cap T_5 = \emptyset$, implying $b_4 < b_5$, which
illustrates that the hooks and slices do in fact generalize the row
(and column) relations.

\begin{cor}\label{cor.hook_slice_nec}
  Let $T \in \mathcal{T}_n$, and $P = P(T)$ be its order poset.  If
  $t_i$ hooks $t_j$, then $i <_P j$.  If $t_i$ slices $t_j$, then $i
  >_P j$.
\end{cor}

\section{Statement of the Main Theorems}

In this section we state some theorems regarding the structure of
ESTs.  Proofs are deferred to
Sections~\ref{sec.structure}--\ref{sec.quant}.  The first theorem
classifies the configuration of ESTs, and answers the question, {\it
  Under what conditions is a sequence of points an EST?}
\begin{theorem}[EST Structure Theorem]\label{thm.EST_structure}
  A sequence of points, $T = (t_i)_{i \leq n}$, is an EST if and only
  if for each $1 \leq i \leq n$, $t_i \in W(T_{i-1})$.
\end{theorem}

Now suppose $T \in \mathcal{T}_n$, and we find its order poset $P =
P(T)$.  Consider a linear extension $A$ of $P$, and consider its EST,
$T' = T(A)$.  A priori, it may not be the case that $T = T'$, however
the following theorem will prove that this is indeed the case.  For a
poset $P = (X, <_{P})$, let $L(P)$ be the set of all linear orders on
the set $X$ that extend $<_{P}$.
\begin{theorem}[Linear Extension Theorem]\label{thm.bijection}
  For any $T \in \mathcal{T}_n$, $[T] = L(P(T))$.
\end{theorem}

Next we strengthen Cor.~\ref{cor.hook_slice_nec} by showing that the
hook and slice relations are sufficient to generate $P(T)$.
\begin{theorem}\label{thm.hook_slice_generate}
  Given $T \in \mathcal{T}$, the order poset $P(T)$ is equal to the
  transitive closure of the hook and slice relations determined by
  $T$.
\end{theorem}

Finally, towards understanding the number of ESTs of size $n$, we find
asymptotic bounds on $|\mathcal{T}_n|$.
\begin{theorem}\label{thm.stronger_bound}
  $|\mathcal{T}_n|=n^{n(1-o(1))}$.
\end{theorem}

\section{EST Structure}\label{sec.structure}

In this section, we examine the geometric structure of an EST, leading
to a proof of Theorem~\ref{thm.EST_structure}.  In what follows we
assume $A = (a_i)_{i \leq n}$ is a sequence of distinct numbers, and
$T = T(A) = (t_i)_{i \leq n} = ((a_i^+, a_i^-))_{i \leq n}$ is its
EST.  The first geometric result, which we call No-Backwards-Placement
simply states that higher-index points cannot be in the shadow of
lower-index points of the EST.

\begin{lemma}[No-Backward-Placement]\label{lem.no-back}
  If $i < j$, then $t_j \notin S_i$.
\end{lemma}
\begin{proof}
  If $t_j \in S_i$, then there exists $t_k \in T_i$ such that either
  $(a_j^+ < a_k^+) \wedge (a_j^- = a_k^-)$, or $(a_j^- < a_k^-) \wedge
  (a_j^+ = a_k^+)$.  It is sufficient to consider only the first case,
  as duality would take care of the second.  Since $t_j$ is to the
  left of $t_k$ in the same row, we have $a_j < a_k$.  However,
  considering the subsequence obtained by appending $a_j$ to the
  longest decreasing subsequence ending at $a_k$, we find that $a_j^-
  \geq a_k^- + 1$, a contradiction.
\end{proof}

We now prove Theorem~\ref{thm.EST_structure}.

\begin{proof}
  The {\it only if} direction is proven by cases.  The case $n=0$ is
  clear by definitions.  It remains to show that for $n \geq 1$ we
  have $t_n \in W_{n-1}$.  By No-Backward-Placement, $t_n \notin
  S_{n-1}$, so it suffices to show either $\mathrm{Row}(a_n^+-1,
  a_n^--1) \cap T_n \neq \emptyset$ or $\mathrm{Col}(a_n^+-1, a_n^--1)
  \cap T_n \neq \emptyset$. It is clear that if $a_n < a_i$ for all $i
  < n$, then $t_n$ must occupy the extreme upper left point of
  $W_{n-1}$. Dually, if $a_n > a_i$ for all $i < n$, then $t_n \in
  W_{n-1}$ at the extreme lower right point.  Now suppose that $a_n^+
  > 1$ and $a_n^- > 1$.  This implies there is some $i < n$ such that
  $a_n^+ = a_i^+ + 1$ (take $a_i$ to be the penultimate element of the
  longest increasing subsequence ending at $a_n$), and $j$ such that
  $a_n^- = a_j^- + 1$ (take $a_j$ analogously in the longest
  decreasing subsequence ending at $a_n$).  Suppose
  $\mathrm{Col}(a_i^+, a_j^-) \cap T_{n} = \emptyset$.  Then it must
  be the case that $a_j^+ \geq a_n^+$, since otherwise $a_j < a_n$ by
  Lemma~\ref{lem.hook_slice}, a contradiction.  Thus $t_j \in
  \mathrm{Row}(a_i^+, a_j^-) \cap T_n$.  Analogously, if we assumed
  that $\mathrm{Row}(a_i^+, a_j^-) \cap T_n = \emptyset$, then
  $\mathrm{Col}(a_i^+, a_j^-) \cap T_n \neq \emptyset$ is forced.  In
  either case, $t_n \in W_{n-1}$.

  For the {\it if} direction, use induction on $n$.  Suppose $T_{n-1}$
  is the EST of the sequence $A' = (a_i)_{i < n}$, and let $t \in
  W_{n-1}$. We wish to show that there is a number $a_n$ such that the
  $n^{th}$ point of $T(A)$ is $t$, where $A = (a_i)_{i \leq n}$.  If
  $t$ is in the extreme upper left of $W_{n-1}$, then choose $a_n <
  a_i$ for all $i < n$; dually, if $t$ is in the extreme lower right,
  then choose $a_n > a_i$ for all $i < n$.  Now if $a_n$ is any
  number, and $a_n$ is increased slightly so that it swaps position
  relative to only one value $a_i$, then $a_n^+$ either remains the
  same or increases by one, and $a_n^-$ either remains the same or
  decreases by one (each case depending on whether $a_i$ was included
  in the longer monotone subsequence).  Thus, the point $t_n$ of the
  EST would either remain unchanged, move to the right one unit, move
  down one unit, or move diagonally down and right one unit each.  We
  have shown above that $t_n$ must always be in the wall of $T_{n-1}$,
  therefore as the value of $a_n$ varies from lesser than all $a_i$ to
  greater than all $a_i$, every point of $W_{n-1}$ must be visited.
  Thus, there exists a value $a_n$ so that the given point $t$ is the
  $n^{th}$ point of the EST $T(A)$.
\end{proof}

\section{Order Poset Structure}\label{sec.P(T)}

\subsection{The Edge Chain}\label{sub.edge_chain}

Let $A = (a_i)_{i \leq n}$ be a sequence of distinct numbers.  In
order to prove Theorem~\ref{thm.bijection}, we develop an inductive
method of constucting both the EST $T(A)$ and its order poset $P(T(A))$.
Along the way, we find that the two structures are in one-to-one
correspondence.

\begin{lemma}\label{lem.edge_chain}
  The placement of points in $T = T_n$ defines a total order $<_{T_n}$
  on the set of points $E_{n-1} \cup \{t_n\}$ such that $t_i <_{T_n}
  t_j \Rightarrow i <_{P(T)} j$.
\end{lemma}
\begin{proof}
  The proof is by induction.  We may assume a total order
  $<_{T_{n-1}}$ on $E_{n-2} \cup \{t_{n-1}\}$.  By
  Lemma~\ref{lem.structure_lemmas}(b), $E_{n-1} \subseteq E_{n-2} \cup
  \{t_{n-1}\}$ so $E_{n-1}$ inherits the total order $<_{T_{n-1}}$.
  As part of the inductive hypothesis, assume that the ordering of
  $E_{n-1}$ is the clockwise order of points.  By ``clockwise order'',
  we mean that we follow the path of wall points starting from upper
  left to lower right, listing the edge points corresponding to each
  wall point, and in the event that two edge points correspond to a
  single wall point ({\it i.e.}, the wall point at which the wall
  changes direction from down to right), then follow the rule, {\it
    upper-left before lower-right}.  As an example, the clockwise
  order of edge points in Fig.~\ref{fig.EST_example} is $(1,1), (2,4),
  (3,3), (4,3), (4,2), (3,1)$.  Let $A$ be any sequence such that
  $T_{n} = T(A)$, and write $A = A' \cup \{a_n\}$.  Observe $T_{n-1} =
  T(A')$.  By Theorem~\ref{thm.EST_structure}, $t_n = (a_n^+, a_n^-)
  \in W_{n-1}$.  We extend the relation $(E_{n-1}, <_{T_{n-1}})$ to a
  relation $(E_{n-1} \cup \{t_n\}, <_{T_{n}})$ based on the position
  of $t_n$ in the wall. There are five cases to consider:
  \begin{enumerate}[1.]
    \item If $t_n$ occupies the extreme upper left point of $W_{n-1}$,
      then the column relation shows $a_n < a_i$, where $t_i$ is the
      least element of $E_{n-1}$.  Thus $t_n$ is the least element of
      $(E_{n-1} \cup \{t_n\}, <_{T_n})$.
    \item Dually, if $t_n$ occupies the extreme lower right point of
      $W_{n-1}$, then $t_n$ is the greatest element of $(E_{n-1} \cup
      \{t_n\}, <_{T_n})$.
    \item If $a_n^+ > 1$ and $\mathrm{Col}(a_n^+-1,a_n^--1)\cap T_n =
      \emptyset$, then consider the edge points $t_i, t_j$ such that
      $a_i^+ = a_n^+ - 1$ and $a_j^+ = a_n^+$.  Clearly $t_i$ and
      $t_j$ are consecutive in the order $<_{T_{n-1}}$.  Then $t_i$
      hooks $t_n$, implying $a_i < a_n$, while $a_j > a_n$ due to the
      column relation.  Extend $<_{T_{n-1}}$ by inserting $t_n$
      between $t_i$ and $t_j$.  
    \item Dually, if $a_n^- > 1$ and $\mathrm{Row}(a_n^+-1, a_n^--1)
      \cap T_n = \emptyset$, then we find consecutive $t_i, t_j$ such
      that $a_i^- = a_n^-$ and $a_j^- = a_n^--1$.  Extend the linear
      order by inserting $t_n$ between $t_i$ and $t_j$.
    \item If none of the above conditions apply, then there exists
      $t_i \in E_{n-1}$ in the same row as $t_n$ (to the left), and
      $t_j \in E_{n-1}$ in the same column (below).  As conditions 3
      and 4 do not apply we must have $a_j^+=a_i^++1$ and
      $a_j^-=a_i^--1$.  Since $a_i < a_n < a_j$ is forced, and $t_i$
      and $t_j$ are consecutive, the linear order is extended by
      inserting $t_n$ between $t_i$ and $t_j$.
  \end{enumerate}
  Note that in all cases the new set of edge points, $E_n \subseteq
  E_{n-1} \cup \{t_n\}$, satisfies the condition that $<_{T_n}$ is the
  clockwise order.
\end{proof}

\begin{cor}
  Within the Order Poset $P(T)$, there are chains $\mathcal{E}_i$
  corresponding to the total orders $(E_i, <_{T_{i}})$ of
  Lemma~\ref{lem.edge_chain}, for each $i \leq n$.  Call
  $\mathcal{E}_n$ the {\bf edge chain} of $T$.
\end{cor}

\subsection{The Order Poset Theorem}

\begin{lemma}\label{lem.edge_relations}
  Let $T \in \mathcal{T}_n$.  The set of order relations
  $\{<_{T_i}\}_{i \leq n}$ from Lemma~\ref{lem.edge_chain} is
  necessary and sufficient for a sequence $A$ to generate $T$.
\end{lemma}
\begin{proof}
  Necessity is clear from Lemma~\ref{lem.edge_chain}.  To show
  sufficiency, we must show that $T$ can be reconstructed from the
  relations.  By induction, it suffices to show that $t_n$ can be
  uniquely determined, given $T_{n-1}$ and the order relation
  $<_{T_n}$.  By Lemma~\ref{lem.structure_lemmas} and
  Theorem~\ref{thm.EST_structure}, there are exactly $|E_{n-1}| + 1$
  distinct points at which $t_n$ can be placed, and each placement
  induces a distinct linear ordering on $E_{n} \subseteq E_{n-1} \cup
  \{t_n\}$.  Thus the placement of $t_n$ is determined by $(E_{n-1}
  \cup \{t_n\}, <_{T_n})$.
\end{proof}

\begin{cor}\label{cor.hook_slice_suf}
  The set of hook and slice relations are sufficent to generate
  $P(T)$.
\end{cor}
\begin{proof}
  By Lemma~\ref{lem.edge_relations}, $P(T)$ is equal to the transitive
  closure of the relations $\{<_{T_i}\}_{i \leq n}$ determined by $T$.
  All of these relations are either hooks or slices.
\end{proof}

\begin{cor}\label{cor.lin_ext}
  Let $T = T_n \in \mathcal{T}_n$.  A sequence $A$ of the numbers
  $\{1, 2, \ldots, n\}$ generates $T$ if and only if $(A, <)$ is a
  linear extension of $([n], <_{P(T)})$.  A real number sequence $A$
  generates $T$ if and only if there is an order-preserving map of
  sequences $A \to A'$ such that $A'$ is a linear extension of $([n],
  <_{P(t)})$.
\end{cor}

Cor.~\ref{cor.hook_slice_suf} implies
Theorem~\ref{thm.hook_slice_generate}, while Cor.~\ref{cor.lin_ext}
implies Theorem~\ref{thm.bijection} .

\subsection{Long chains within the order poset}\label{sub.long_chain}

With a view towards counting $\mathcal{T}_n$, we find a chain in $P =
P(T)$ (for arbitrary $T \in \mathcal{T}_n$) which includes the edge
chain $\mathcal{E}_{n}$ as a subchain.  Write the corner points in
linear order: $C(T) = \{t_{r_1}, t_{r_2}, \ldots, t_{r_k}\}$.  Call
two corner points, $t_{r_{i}}$ and $t_{r_{i+1}}$, {\bf consecutive}.

\begin{figure}[hb!]
\begin{center}
  \scalebox{0.9}{%
      \includegraphics{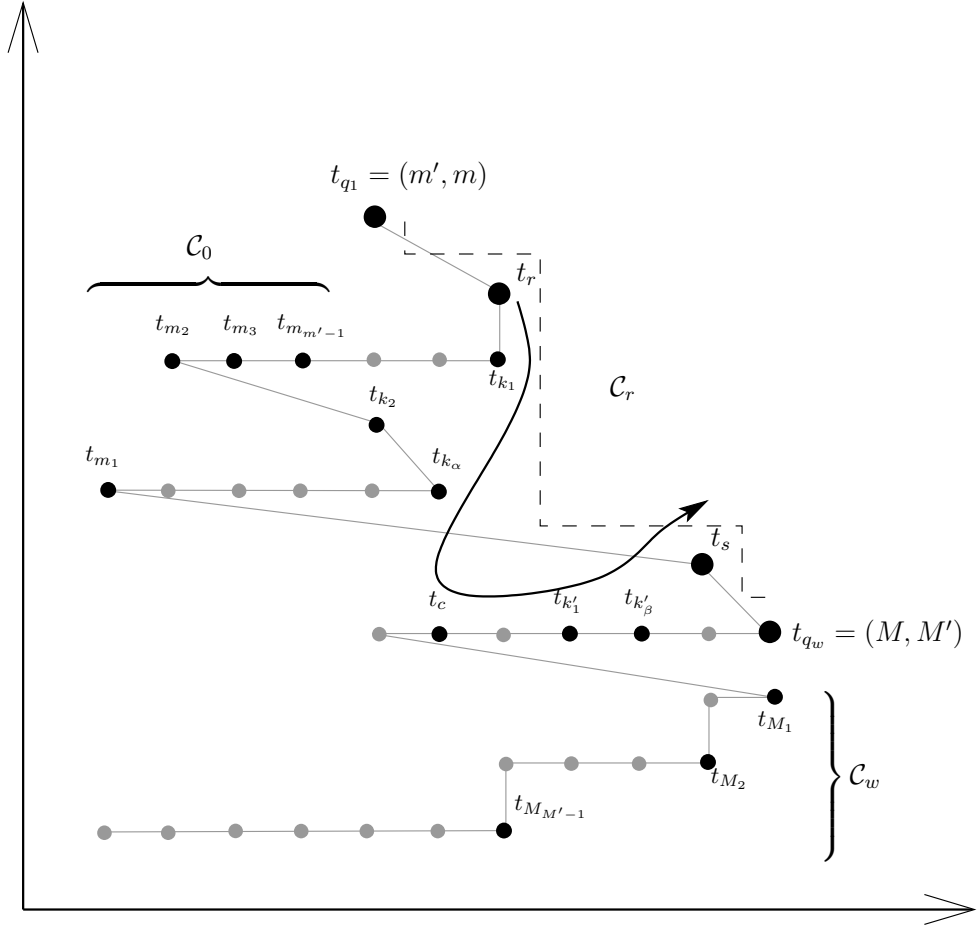}
  }
  \caption{Chains in $P(T)$.  The diagram shows an example EST, with
    points labeled as in the proof of Lemma~\ref{lem.small_chain}.
    Corner points are shown as larger dots.  The curved arrow
    indicates the direction of points in the chain $\mathcal{C}_r$.}
  \label{fig.chains}
\end{center}
\end{figure}

\begin{lemma}\label{lem.small_chain}
  Let $t_r = (x_r, y_r)$ and $t_s = (x_s, y_s)$ be consecutive corner
  points of $T$.  There is a chain from $r$ to $s$ in $P$ with at
  least $(y_r-y_s)+(x_s-x_r)-1$ intermediate points.
\end{lemma}
\begin{proof}
  Let $\mathcal{C}_r$ be the portion of the edge chain strictly
  between $t_r$ and $t_s$.  Note that $|\mathcal{C}_r| = \alpha +
  \beta$, where $\alpha = x_s-x_r-1$ and $\beta = y_r-y_s-1$.  In
  terms of the poset $P$, we may write:
  \[
    \mathcal{C}_r = {k_1} <_P {k_2} <_P \cdots <_P {k_\alpha} <_P
            {k'_1} <_P {k'_2} <_P \cdots <_P {k'_\beta},
  \]
  where an index $k_i$ refers to an edge point $t_{k_i}$ corresponding
  to a vertical portion of the wall, while an index $k'_j$ refers to
  an edge point $t_{k'_i}$ corresponding to a horizontal portion of
  the wall. It may be helpful to refer to Fig.~\ref{fig.chains}.  We
  shall find an additional point $t_c$ such that $k_{\alpha} <_{P} c
  <_{P} k'_1$.  Without loss of generality, assume $k_{\alpha} <
  k'_1$, since the case $k_{\alpha} > k'_1$ can be handled in $T^*$.
  Set $t_c = (x_c, y_{k'_1})$ where $x_c=\max\{x \;|\; (x, y_{k'_1})
  \in T_{k_{\alpha}} \}$.  Such a point $t_c$ exists by the EST
  Structure Theorem, and No-Backwards-Placement implies $x_c<
  x_{k'_1}$.  Now $c <_{P} k'_1$ by the horizontal relation, and $c
  >_{P} k_{\alpha}$ because $\mathrm{Row}(t_c) \cap T_{k_{\alpha}} =
  \emptyset$ implies that $t_c$ slices $t_{k_{\alpha}}$.
\end{proof}
\begin{lemma}\label{lem.chain}
  The poset $P(T)$ has height at least $M + m - 1$, where
  \begin{equation}\label{eqn.def_M_m}
    M = M(T) = \max\{x \;|\; (x,y) \in T\} \qquad \textit{and} \qquad
    m = m(T) = \max\{y \;|\; (x,y) \in T\}.
  \end{equation}
\end{lemma}
\begin{proof}
  Let $\mathcal{L}$ be the chain obtained from the edge chain by
  inserting each internal point $t_c$ as in
  Lemma~\ref{lem.small_chain}.  There are exactly $|C_n| - 1$ such
  internal points.  A simple counting argument gives the relation
  $|E_n| = M + m - |C_n|$, and the result follows from $|\mathcal{L}|
  = |E_n| + |C_n| - 1$.
\end{proof}

\begin{prop}\label{prop.height}
  For all $T \in \mathcal{T}_n$, $\mathrm{Height}(P(T)) \geq
  2\sqrt{n}-1$.
\end{prop}
\begin{proof}
  Let $M$ and $m$ be as in Eqn.~(\ref{eqn.def_M_m}).  Since the points
  of $T$ must fall in the region $[1,M] \times [1,m]$ we have $Mm \geq
  n$. Minimizing $M + m$ over this (integer) constraint yields $M + m
  \geq \lceil\sqrt{n}\rceil + \lfloor\sqrt{n}\rfloor$ which is well
  approximated by $2\sqrt{n}$. The result follows from
  Lemma~\ref{lem.chain}.
\end{proof}

\section{Quantitative Results}\label{sec.quant}

Our main goal in this section is to find bounds on the order of
$\mathcal{T}_n$, but we will also discuss the distribution of the
sizes of the equivalence classes $[T] = L(P(T))$.  We will close with
some explicit values for small $n$.  

First observe that Prop.~\ref{prop.height} can be used to get a crude
lower bound on the number of linear extensions of a given poset $P =
P(T)$ and, by extension, $|\mathcal{T}_n|$.  Suppose $P$ has height $h
= 2\lceil \sqrt{n}\rceil - 1$ (recall all such posets have minimum and
maximum elements).  Among posets (of this type) the number of linear
extensions decreases with the number of additional relations and so to
maximize the size of the equivalence class, $\Delta_n = \max\{|[T]|
\;|\; T\in\mathcal{T}_n\}$, the $(n-h)$ points outside the unique
maximal chain must form an antichain. To extend $P$ to a linear order
there are $(n-h)!$ ways to order the antichain (this produces a poset
consisting of 2 internally disjoint chains) and then
$\binom{n-2}{h-2}$ ways of interlacing the chains. Hence $\Delta_n\leq
(n-h)!\binom{n-2}{h-2}$ from which one could get an explicit lower
bound $|\mathcal{T}_n| \geq \frac{n!}{\Delta_n}$.  However, this bound
is very weak as the distribution of class sizes is skewed towards the
low end (see Fig.~\ref{fig.EST-dist8}). We give a stronger bound in
Theorem~\ref{thm.stronger_bound} which shows that the order of
$\mathcal{T}_n$ grows exponentially.  The following is a proof of
Theorem~\ref{thm.stronger_bound}.

\begin{figure}[ht!]
\begin{center}
  \includegraphics[scale=1.1]{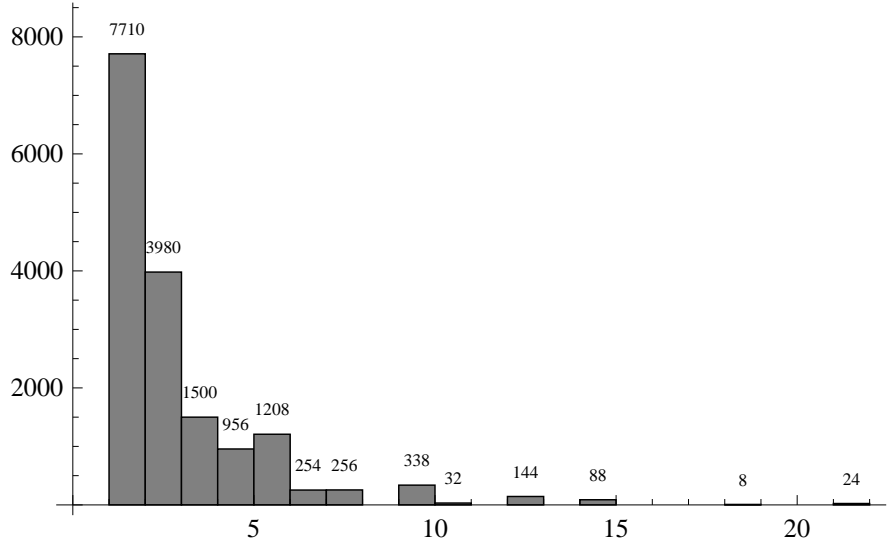}
  \caption{Histogram: Frequency of $|[T]|$ for $T \in
    \mathcal{T}_8$} \label{fig.EST-dist8}
\end{center}
\end{figure}

\begin{proof}
  The upper bound is immediate as $|\mathcal{T}_n|\leq n!$. For the
  lower bound consider the tableau
  \[
    \{(1,1),(1,2),(1,3)\ldots (1,k)\}.
  \]
  Observe that any extension has a wall of size at least $k$ as the
  number of columns cannot be decreased. Sequentially choosing the
  remaining $n-k$ points we get that $|\mathcal{T}_n|\geq
  k^{n-k}$ for all $k\leq n$. Taking $k=\frac{n}{\log
    n}$ gives the desired result.
\end{proof}

\begin{figure}[ht!]
  \begin{center}
    \begin{tabular}{|c|c|c|c|c|}
      \hline
      $n$      & $|\mathcal{T}_n|$ & $\Delta_n$ &
      $U_n$  \\ \hline
      $\leq 3$ &    $n!$ &      1 &  1     \\ \hline
      4        &      22 &      2 & .83    \\ \hline
      5        &      96 &      3 & .78    \\ \hline
      6        &     480 &      5 & .68    \\ \hline
      7        &    2682 &     10 & .57    \\ \hline
      8        &   16498 &     21 & .47    \\ \hline
      9        &  110378 &     44 & .44    \\ \hline
      10       &  795582   &      &        \\ \hline
      11       & 6131722   &      &        \\ \hline
      12       & 50224736  &      &        \\ \hline
      13       & 434989688 &      &        \\ \hline

    \end{tabular}
    \caption{ Values of $|\mathcal{T}_n|$, $\Delta_n=\max\{|[T]| \;|\;
      T\in\mathcal{T}_n\}$, and $U_n = \frac{1}{|\mathcal{T}_n|}|\{T
      \in\mathcal{T}_n \;|\; |[T]| = 1\}|$.}
    \label{fig.dist_table}
  \end{center}
\end{figure}

\begin{rmk}
 It may be possible to derive a more explicit lower bound on the
 number of EST's by taking into account the distribution of sizes of
 equivalence classes in $\mathcal{T}_n$ for arbitrary $n$.
 Unfortunately, we do not know much about the distribution except that
 it resembles a power law (with negative exponent).  Even knowing the
 number of ESTs that have unique realizations (equivalently, that have
 linear posets) would be helpful.  We present some data in
 Fig.~\ref{fig.dist_table}.
\end{rmk}

\begin{acknowledgement*}
We would like to thank the referees for many good suggestions, in
particular, making the proof of Thm.~\ref{thm.EST_structure} much
clearer, shortening the proofs of Thms.~\ref{thm.bijection}
and~\ref{thm.hook_slice_generate}, and strengthening
Thm.~\ref{thm.stronger_bound}.  Moreover, we wish to thank Dwight
Duffus for helpful conversations.
\end{acknowledgement*}


\bibliographystyle{plain}


\end{document}